\documentclass[onefignum,onetabnum]{siamart171218}

\usepackage{longtable,booktabs,array}  

\usepackage[utf8]{inputenc}
\usepackage{ulem,bm}
\usepackage{amsmath}
\usepackage{amssymb}
\usepackage{changepage}
\usepackage{caption}
\usepackage{graphicx}

\title{Diagonally forced systems and the spectral signature of matrix cycles}
\author{Michael A.S. Thorne \\ British Antarctic Survey}
\date{}

\begin{document}
\maketitle

\begin{abstract}
Given any square matrix, $\mathbf{M}$, whose diagonal elements are negative, and which are multiplied by a variable, $\sigma$, we wish to find the minimal $\sigma$ such that the eigenvalue of $\mathbf{M}_{\sigma}$ is exactly zero. 
By Gershgorin, we know that $\mathbf{M}_{\sigma}$ can be made stable by making $\sigma$ large enough. We prove a relation which analytically determines when and how we are able to find the value of $\sigma$ such that the maximal eigenvalue is exactly zero. In so doing, we prove the equivalence of the roots of the characteristic polynomial of $\mathbf{M}_{\sigma}$ and the eigenvalues that arise from a scaling operation on $\mathbf{M}$. Further, through the characteristic polynomial, we are able to isolate the dominant feedback cycles comprising the elements of the matrix which, under the action of $\sigma$, ensures the stability of the system. We then explore, using the stabilising and destabilising cycles within the coefficients of the characteristic polynomial, an intrinsic spectral signature associated with any square matrix based on the sign (or zero) of its respective elements.
\end{abstract}



\section{Introduction} \label{intro-sec}
The use of a variable as a multiplier, acting on the diagonal values of a Jacobian, has long played a role as a means of perturbing a system [1] and as a mechanism to understand its stability [2]. However, the consequences of forcing the diagonal have not been analytically explored. The question we are concerned with is whether we can analytically find the exact value of the multiplying factor in order to be able to force the maximal (largest real-part) eigenvalue of the Jacobian to be zero, that is, the stability tipping point of the system under the action of diagonal forcing. In the process of resolving this question, we prove a relation between the multiplying factor and the eigenvalues that are determined through a scaling mechanism in which the absolute values of the diagonal of the Jacobian are used to divide through each respective row, and the resulting matrix subsequently translated by the identity. This relation leads to an explanation not only of why and when one can produce the exact value needed for stability, but also which feedbacks in the system contribute the most to the stability when driven by diagonal forcing. We then consider a construction based solely on the weighting of the matrix elements within the characteristic polynomial coefficients, which can be used to assign an intrinsic spectral signature to every matrix based solely on its structure. We set the stage in the following section by looking at the characteristic polynomial of systems with a multiplying factor on their diagonal. 
\vspace{0.5cm}
\section{Background}

Let $\mathbf{A}$ be an $n \times n$ matrix with real-valued entries. 
Then, multiplying the diagonal entries $\mathbf{D}$ of $\mathbf{A}$ ($\mathbf{D}_{\mathbf{A}}$) by a variable $\sigma \in \mathbb{C}$, we generate the matrix $\mathbf{A}_{\sigma}$, \\

\begin{equation}
\mathbf{A}_{\sigma}=\mathbf{A}-(1-\sigma)\cdot \mathbf{D}_{\mathbf{A}}
\label{eq:sigma}
\end{equation}\\

The characteristic polynomial, $\Pi$, of $\mathbf{A}_{\sigma}$ ($\Pi_{\mathbf{A}_{\sigma}}$) is a monic polynomial in $x$, each of whose coefficients is a polynomial in $\sigma$:

\begin{equation}
  \Pi_{\mathbf{A}_{\sigma}} = \sum_{i=0}^n p_i x^i = \sum_{i=0}^n (\sum_{j=0}^{n-i} q_{i,j}\sigma^j) x^i
\label{eq:char}
\end{equation}
with $p_n=1$ and 

\begin{equation*}
p_i = \sum_{j=0}^{n-i} q_{i,j}\sigma^j
\label{eq:pi}
\end{equation*}\\
where the $q_{i,j}$ consist of signed terms of products of $n-i$ entries of $\mathbf{A}$.\\ 

We now want to consider certain constraints on $\mathbf{A}$. For this we create the following definition:   

\begin{definition}
If a $\sigma$ ($\operatorname{Re}(\sigma) > 0$) can be found for $\mathbf{A}_{\sigma}$ such that the maximal (largest real-part) eigenvalue ($\lambda_{max(\mathbf{A}_{\sigma})}$) is zero and any larger real-part of $\sigma$ results in $\operatorname{Re}(\lambda_{max(\mathbf{A}_{\sigma})}) < 0$, then $\mathbf{A}_{\sigma}$ is denoted $\sigma$-stable.
\end{definition}

One way to ensure that the system is $\sigma$-stable is if each of the sums of the $\binom{n}{i}$ sets of the negative of the diagonal elements (${a_{i,i}} \in \mathbf{A}, \forall i$),
\begin{equation*}
\{\sum_{i=1}^n (-a_{i,i}),\sum_{i \neq j}^n (-a_{i,i})(-a_{j,j}),\sum_{i \neq j \neq k}^n (-a_{i,i})(-a_{j,j})(-a_{k,k}),...\},
\end{equation*}\\
are positive, then the highest order of $\sigma$ in each $p_i$ will be positive. If the condition of each $p_i$ having a positive leading coefficient is satisfied, then by Gershgorin [3], $\mathbf{A}_{\sigma}$ can be made stable by ensuring a $\sigma$ ($\operatorname{Re}(\sigma) > 0$) is applied whose real-part is large enough.

\begin{theorem}\label{theorem}
Let $\mathbf{A}_{\sigma}$ be $\sigma$-stable. The real-valued roots 
$(\{\sigma_{i,1},...,\sigma_{i,k}\} \subset \mathbb{R})$ of each $p_i$ indicate the point of respective coefficient sign change in the characteristic polynomial given the value of $\sigma$ of $\mathbf{A}_{\sigma}$. Furthermore, the largest real-valued root for each $p_i$ indicates when the coefficient becomes positive and it remains positive for any larger (real-part of) $\sigma$ applied to $\mathbf{A}_{\sigma}$.   
\end{theorem}

\begin{proof}
Each coefficient of the characteristic polynomial in $x$, which is in the field of $\mathbb{R}$, is itself a polynomial, $p_i$, in $\sigma$ and, by Descartes' rule of signs [4], the real roots reflect the changing sign of the $p_i$ and therefore of the coefficients of the characteristic polynomial in $x$. 

A necessary condition for stability of a system is that all the coefficients of the characteristic polynomial are positive. By the diagonal conditions and definition of $\sigma$-stable matrices, they can be made stable by increasing (the real-part of) $\sigma$. Therefore, any $\operatorname{Re}(\sigma)$ larger than the largest real root of each $p_i$ ensures it is a positive coefficient. 
\end{proof}

While a necessary condition for stability is that all the coefficients of a characteristic polynomial are positive, it is not sufficient. This point, along with Theorem~\ref{theorem}, leads to the following result. For this, we define the set 
\begin{equation*}
\Omega=\{\sigma_{0,1},...,\sigma_{0,k_0},...,\sigma_{i,1},...,\sigma_{i,k_i},...,\sigma_{n-1,1},...,\sigma_{n-1,k_{n-1}}\}
\end{equation*}
consisting of all the real-valued roots ($\Omega \subset \mathbb{R}$) of all the $p_i$ $(i=0,1,...,n-1)$ polynomials from the characteristic polynomial of a given $\mathbf{A}_{\sigma}$.

\begin{theorem}
If $\exists$ a $\sigma \in \mathbb{R}$ such that the maximal eigenvalue, $\lambda_{max(\mathbf{A}_{\sigma})}$ of $\mathbf{A}_{\sigma}$, is zero, 
then $\sigma = max(\Omega)$.
\label{lemreal}
\end{theorem}

\begin{proof}
By the necessary condition of stability, all coefficients must be positive, which means that $\sigma \ge max(\Omega)$. 
By the proof of Obrechkoff [5], if a polynomial has positive coefficients and has a root in the right-half plane, then the root cannot lie on the real axis. This implies that if $\sigma \in \mathbb{R}$ then $\sigma \leq max(\Omega)$. Therefore, $\sigma = max(\Omega)$.
\end{proof}

A consequence of Theorem \ref{lemreal} is that if all the coefficients of $\mathbf{A}_{\sigma}$ are positive, and if there exists a positive root of $\mathbf{A}_{\sigma}$, then it is necessarily complex-valued. 
\vspace{0.5cm}
\section{Roots and spectrum equivalence of diagonal forcing and scaling}

 We now focus on Jacobians whose diagonal elements are all negative. Let $\mathbf{M}$ be a real valued square matrix with $m_{i,i} < 0$, 
 $\forall i$, (therefore $\prod_1^n m_{i,i} \neq 0$), 
 and where all other $m_{i,j}, i \neq j$, have either sign, or zero. We restate Eq. \ref{eq:sigma}, again with the variable $\sigma \in \mathbb{C}$,
 
\begin{equation*}
\mathbf{M}_{\sigma}=\mathbf{M}-(1-\sigma)\cdot \mathbf{D}_{\mathbf{M}}.
\label{eq:sigmaMP}
\end{equation*}\\

Since all the diagonal elements in $\mathbf{M}$ are negative, for simplicity of expression we can alter this expression to reflect the absolute value of the diagonal elements, 

\begin{equation}
\mathbf{M}_{\sigma}=\mathbf{M}+(1-\sigma)\cdot |\mathbf{D}_{\mathbf{M}}|.
\label{eq:sigmaM}
\end{equation}\\

Similarly to Eq. \ref{eq:char}, we want to define the characteristic polynomial, $\Pi$, of $\mathbf{M}_{\sigma}$ ($\Pi_{\mathbf{M}_{\sigma}}$) where, 

\begin{equation*}
  \Pi_{\mathbf{M}_{\sigma}} = \sum_{i=0}^n r_i x^i = \sum_{i=0}^n (\sum_{j=0}^{n-i} s_{i,j}\sigma^j) x^i
\label{eq:charM}
\end{equation*}
with $r_n=1$ and 

\begin{equation}
r_i = \sum_{j=0}^{n-i} s_{i,j}\sigma^j.
\label{eq:piM}
\end{equation}\\

We now construct a hollow matrix, by dividing each row by the absolute value of its diagonal element and then translating it by the identity [2],

\begin{equation}
\bar{\mathbf{M}}_0 = |\mathbf{D}_{\mathbf{M}}|^{-1} \mathbf{M}+\mathbf{I}.
\label{eq:scale}
\end{equation}

Through $\Pi_{\mathbf{M}_{\sigma}}$, we can relate the spectrum of $\bar{\mathbf{M}}_0$ to values of $\sigma$ which generate zeros of  $\mathbf{M}_{\sigma}$.

\begin{theorem}
The characteristic polynomial of $\bar{\mathbf{M}}_0$ ($\Pi_{\bar{\mathbf{M}}_0}$) in $x$ is the same polynomial as the last coefficient $r_0$ of $\Pi_{\mathbf{M}_{\sigma}}$ in $\sigma$ ($\Pi_{\bar{\mathbf{M}}_0} = r_0$, $r_0 = (-1)^n$ $\det$$(\mathbf{M}_{\sigma})$). Let $\lambda_{\bar{\mathbf{M}}_0}$ be an eigenvalue of $\bar{\mathbf{M}}_0$. Therefore when $\sigma=\lambda_{\bar{\mathbf{M}}_0}$, $r_0$ vanishes, forcing $\mathbf{M}_{\sigma}$ to have an eigenvalue of zero.
\label{charpo}
\end{theorem}
\begin{proof}
We need to show that the characteristic polynomial of $\bar{\mathbf{M}}_0$ ($\Pi_{\bar{\mathbf{M}}_0}$) in $x$ is the same polynomial as the last coefficient $r_0$ of $\Pi_{\mathbf{M}_{\sigma}}$ in $\sigma$, so we rewrite the characteristic polynomial of $\bar{\mathbf{M}}_0$ in terms of $\sigma$:\\
\begin{equation*}
\det(\bar{\mathbf{M}}_0 - \sigma \cdot \mathbf{I})=0
\end{equation*}\\
which we can rewrite using Eq. \ref{eq:scale} as\\
\begin{equation*}
\det(|\mathbf{D}_{\mathbf{M}}|^{-1} \cdot \mathbf{M} + \mathbf{I} - \sigma \cdot \mathbf{I})=0.
\end{equation*}\\
Multiplying det($|\mathbf{D}_{\mathbf{M}}|$) through both sides gives\\
\begin{equation*}
\det(|\mathbf{D}_{\mathbf{M}}|) \cdot \det(|\mathbf{D}_{\mathbf{M}}|^{-1} \cdot \mathbf{M} + \mathbf{I} - \sigma \cdot \mathbf{I})=0 
\end{equation*}\\
which simplifies to\\ 
\begin{equation*}
\det(|\mathbf{D}_{\mathbf{M}}| \cdot (|\mathbf{D}_{\mathbf{M}}|^{-1} \cdot \mathbf{M}+\mathbf{I} - \sigma \cdot \mathbf{I}))=\det(\mathbf{M}+(1-\sigma)\cdot |\mathbf{D}_{\mathbf{M}}|) = 0.
\end{equation*}\\
This last expression can be rewritten using Eq. \ref{eq:sigmaM} as\\  
\begin{equation*}
\det(\mathbf{M}_{\sigma}) = 0, 
\end{equation*}\\
which is the last coefficient of $\Pi_{\mathbf{M}_{\sigma}}$. Therefore, when $\sigma=\lambda_{\bar{\mathbf{M}}_0}$, the last coefficient of $\Pi_{\mathbf{M}_{\sigma=\lambda_{\bar{\mathbf{M}}_0}}}$ vanishes. With $r_0(\sigma) = 0$, the remaining polynomial is\\
\begin{align*}
  & x^n + r_{n-1}(\sigma)x^{n-1}+...+r_2(\sigma)x^2+r_1(\sigma)x = \\
  & \phantom{{}={}} \begin{aligned}[t] x(x^{n-1}+r_{n-1}(\sigma)x^{n-2}+...r_2(\sigma)x+r_1(\sigma)).\\
 \end{aligned}
\end{align*}\\
Pulling out an $x$ shows clearly that there is a root of zero, with the remaining polynomial in $x$ of degree $n-1$.
\label{scale}
\end{proof}
\begin{corollary}
Let $\lambda_{max(\bar{\mathbf{M}}_0)}$ be the maximal (real-part) eigenvalue of $\bar{\mathbf{M}}_0$. \\
When $\sigma=\lambda_{max(\bar{\mathbf{M}}_0)}$, $r_0(\sigma)$ of $\Pi_{\mathbf{M_{\sigma}}}$ vanishes, forcing $\mathbf{M}_{\sigma}$ to have an eigenvalue of zero. Since $\lambda_{max(\bar{\mathbf{M}}_0)}$ is the maximal eigenvalue of $\bar{\mathbf{M}}_0$, then $\sigma=\lambda_{max(\bar{\mathbf{M}}_0)}$ is the maximal zero of $r_0(\sigma)$.
\end{corollary}

\begin{proof}
By Theorem~\ref{charpo}, $\Pi_{\bar{\mathbf{M}_0}}$ in $x$ and $r_0$ of $\Pi_{\mathbf{M_{\sigma}}}$ in $\sigma$ are equivalent polynomials. Therefore, the maximal zero of  
 $\Pi_{\bar{\mathbf{M}_0}}$ is the maximal zero of $r_0$.
\end{proof}

The above shows that when an eigenvalue of the hollow scaled matrix $\bar{\mathbf{M}}_0$ is set as $\sigma$, then $\mathbf{M}_{\sigma}$ produces a zero eigenvalue. This does not, however, ensure that the zero of $\mathbf{M}_{\sigma=\lambda_{max(\bar{\mathbf{M}}_0)}}$ is maximal. There are degenerate cases where the zero of $\mathbf{M}_{\sigma=\lambda_{max(\bar{\mathbf{M}}_0)}}$ is not maximal and there is at least one positive (real-part) eigenvalue, which is due to the value of $\sigma$ in relation to the roots of the other $r_i$ ($i \neq 0$), which play a greater role, under the action of $\sigma$, in driving the system toward stability than $r_0$. 

However, when $\sigma=\lambda_{\max(\bar{\mathbf{M}}_0)} \implies \lambda_{\max(\mathbf{M}_{\sigma})} = 0$, then the minimal value of $\sigma$ for which $\lambda_{\max(\mathbf{M}_{\sigma})} = 0$ has been found, and we have analytically determined the smallest value which can be multiplied by the diagonal elements of $\mathbf{M}$ in order to arrive at its point of stability. 

As a consequence, we have also been able to determine that the length of the critical feedback (that which plays the crucial role in the stability) of the system (under the action of $\sigma$) is the largest possible cycle size (i.e. the lengths of the cycles comprising $r_0$ as described in Eq.~\ref{eq:piM}). By Theorem~\ref{charpo}, we know that 

\begin{equation*}
r_0 = \sum_{j=0}^{n} s_{0,j}\sigma^j = \sum_{j=0}^{n} s_{0,j}\lambda_{\max(\bar{\mathbf{M}}_0)}^j = 0,
\end{equation*}\\

and have therefore been able to isolate the cycles themselves that have played the dominant role ($\sum_{j=0}^{n} s_{0,j}$), under the action of $\lambda_{\max(\bar{\mathbf{M}}_0)}$, that ensure the system becomes stable.\\ 

In the degenerate cases where the scaling process produces a maximal eigenvalue that is positive, although analytically intractable, the critical feedback, under the action of $\sigma$, is now determined to be of size $< n$.

\section{Matrix cycles under the action of $\sigma$}
Previously, the structure of the destabilising cycles, or tipping cycles,  in each coefficient of the characteristic polynomial in various sign restricted forms was explored [6]. We can extend the idea of tipping cycles to systems that are under the action of $\sigma$. Consider the following matrix, 
\begin{equation}
\left[ \begin{array}{rrrr} -a\sigma & b & c & d  \\ -e  & -f\sigma & g & h \\ -i & -j & -k\sigma & l \\ -m&-n&-o&-p\sigma  \end{array} \right]. 
\label{eq:sigmat}
\end{equation}\\
Then the coefficient $r_0$ in $\sigma$ will have the form, \\
\begin{align*}
(afkp)\sigma^4 + (dfkm + ahkn + aflo + cfip + agjp + bekp)\sigma^2 \\ + (bhkm + cflm + agln + bgip - dekn - dfio - ahjo - cejp)\sigma^1 \\ + (dgjm + bglm + chin + dejo + belo - chjm - dgin - celn - bhio) \sigma^0 
\end{align*}

In the coefficient of $\sigma^0$, the number of tipping cycles, or destabilising cycles, is 4, that is $\{chjm,dgin,celn,bhio\}$, while the total number of cycles is 9, so the coefficient sensitivity of $\sigma^0$ is $4/9$. Table 1 shows the coefficient sensitivity for each coefficient for the same sign restricted form (negative lower diagonal and diagonal elements, positive upper diagonal elements) without diagonal forcing up to size 8 (a), along with the same numbers found in $r_0$ (b), but decomposed by the power of $\sigma$. As can be seen, the total summed values of both the denominator and numerator of (b) for a given size is equal to the value of $a_0$ in (a). 

\begin{table}[h]
\caption{\textit{Table a shows the coefficient sensitivity (the number of destabilising cycles, or tipping cycles) for each coefficient for systems that have negative lower diagonal and diagonal entries and positive upper diagonal entries, taken from [6]. Table b is the coefficient sensitivity for the $r_0$ coefficients in $\sigma$. The total values of the denominator and numerator for a given size matrix in Table b is equal to the value of $a_0$ in Table a, but with the values decomposed into powers of $\sigma$.}}
\begin{adjustwidth}{-.8cm}{}
\tiny
\begin{tabular}{c c c c c c c c c c c}
    a) & $n$ & $a_8$$x^8$ & $a_7$$x^7$ & $a_6$$x^6$ & $a_5$$x^5$ & $a_4$$x^4$ & $a_3$$x^3$ & $a_2$$x^2$ & $a_1$$x$ & $a_0$ \\
      & $2$ &&&& & &  & + & + & + \\
      & $3$ & &&&& & + & + & + & 1/6 \\
      &$4$ & &&&& + & + & + & 4/24 & 8/24 \\
      &$5$ &&&& + & + & + & 10/60 & 40/120 & 52/120 \\
      &$6$ & &&+&+&+ & 20/120 & 120/360 & 312/720 & 344/720 \\
      &$7$ & &+&+&+& 35/210 & 280/840 & 1092/2520 & 2408/5040 &  2488/5040 \\
      &$8$ &+&+&+& 56/336 & 560/1680 & 2912/6720 & 9632/20160 & 19904/40320 & 20096/40320 \\
    \\
b) & $n$ & $s_8$$\sigma^8$ & $s_7$$\sigma^7$ & $s_6$$\sigma^6$ & $s_5$$\sigma^5$ & $s_4$$\sigma^4$ & $s_3$$\sigma^3$ & $s_2$$\sigma^2$ & $s_1$$\sigma$ & $s_0$ \\
&$2$  &&&&&  && 0/1 &  & 0/1 \\
& $3$  &&&&&& 0/1 &  & 0/3 & 1/2 \\
& $4$  & &&&& 0/1 & & 0/6 & 4/8 & 4/9 \\
&$5$  &&&& 0/1 &  & 0/10 & 10/20 & 20/45  & 22/44 \\
&$6$  &&&0/1&& 0/14 & 20/40 & 60/135 & 132/264 & 132/266\\
&$7$ & & 0/1 & & 0/21 & 35/70 & 140/315 & 462/924 & 924/1855 & 927/1854 \\
&$8$  &0/1&  & 0/28 & 56/112 & 280/630 & 1232/2464 & 3696/7420 &  7416/14832 & 7416/14833 \\

\end{tabular}
\end{adjustwidth}{}
\normalfont
\label{table:cfs}
\end{table}

In [6], weighted matrices of the elements of the tipping cycle sets were constructed, as a ledger of the element counts. This can be extended to both the tipping cycles (those cycles that destabilise the system) and the positive cycles. For example, the weighted matrix for the tipping cycles in the coefficient of $\sigma^0$ derived from \eqref{eq:sigmat} above is 
\begin{equation*}
\left[ \begin{array}{rrrrr} 0& 1 & 2 & 1  \\ 1 & 0  & 1 & 2  \\ 2 & 1 & 0  & 1 \\ 1&2&1&0  \end{array} \right]. 
\end{equation*}

Generating these matrices for all the positive and negative cycles for all the coefficients of $r_0$ above yields the equation, \\

\resizebox{0.9\hsize}{!}
{$\left[ \begin{array}{rrrrr} 1& 0 & 0 & 0  \\ 0 & 1  & 0 & 0  \\ 0 & 0 & 1  & 0 \\ 0&0&0&1  \end{array} \right] \sigma^4 + \left[ \begin{array}{rrrrr} 3 & 1 & 1 & 1  \\ 1 & 3  & 1 & 1  \\ 1 & 1 & 3  & 1 \\ 1&1&1&3  \end{array} \right] \sigma^2 + \left( \left[ \begin{array}{rrrrr} 1& 2 & 1 & 0  \\ 0 & 1  & 2 & 1  \\ 1 & 0 & 1  & 2 \\ 2&1&0&1  \end{array} \right] - \left[ \begin{array}{rrrrr} 1& 0 & 1 & 2  \\ 2 & 1  & 0 & 1  \\ 1 & 2 & 1  & 0 \\ 0&1&2&1  \end{array} \right] \right) \sigma^1 \\+ \left( \left[ \begin{array}{rrrrr} 0& 2 & 1 & 2  \\ 2 & 0  & 2 & 1  \\ 1 & 2 & 0 & 2 \\ 2&1&2&0  \end{array} \right] - \left[ \begin{array}{rrrrr} 0& 1 & 2 & 1  \\ 1 & 0  & 1 & 2  \\ 2 & 1 & 0  & 1 \\ 1&2&1&0  \end{array} \right] \right)\sigma^0$} \\

\vspace{0.2cm}
A spectral signature could then be associated with this matrix polynomial. Conversely, drawing on Theorem~\ref{charpo} we could reframe the original matrix to be monic solely in $x$. Therefore, scaling the matrix \eqref{eq:sigmat} without the $\sigma$, 

\begin{equation*}
\left[ \begin{array}{rrrr} 0 & b/a & c/a & d/a  \\ -e/f  & 0 & g/f & h/f \\ -i/k & -j/k & 0 & l/k \\ -m/p&-n/p&-o/p&0  \end{array} \right]
\end{equation*}
which, if we rewrite as 
\begin{equation}
\left[ \begin{array}{rrrr} 0 & q & r & s  \\ -t  & 0 & u & v \\ -w & -x & 0 & y \\ -z&-\alpha&-\beta&0  \end{array} \right]
\label{eq:mat}
\end{equation}

we can construct a matrix polynomial, or pencil, $\mathcal{L}$, where
each $D_k$ (which together comprise their D-tuple) is constructed from the difference of the positive and negative weighted cycle matrices for the given coefficient in the characteristic polynomial in $x$. In general,

\[
  \mathcal{L}(x)=I_n x^n+\sum_{k=0}^{n-1} D_k x^k,
  \qquad D_k=P_k-N_k\in\mathbb{Z}^{n\times n}.
\]

For the matrix \eqref{eq:mat} above,  

\begin{equation*}
D_3 = \left[ \begin{array}{rrrrr} 0&0 & 0 & 0  \\ 0 & 0  & 0 & 0  \\ 0 & 0 & 0  & 0 \\ 0&0&0&0  \end{array} \right]
\end{equation*}
\begin{equation*}
D_2 =P_2-N_2=\left[ \begin{array}{rrrrr} 0&1 & 1 & 1  \\ 1 & 0  & 1 & 1  \\ 1 & 1 & 0  & 1 \\ 1&1&1&0  \end{array} \right] - \left[ \begin{array}{rrrrr} 0&0 & 0 & 0  \\ 0 & 0  & 0 & 0  \\ 0 & 0 & 0  & 0 \\ 0&0&0&0  \end{array} \right] = \left[ \begin{array}{rrrrr} 0&1 & 1 & 1  \\ 1 & 0  & 1 & 1  \\ 1 & 1 & 0  & 1 \\ 1&1&1&0  \end{array} \right] 
\end{equation*}

\begin{equation*}
D_1 =  P_1-N_1=\left[ \begin{array}{rrrrr} 0& 2 & 1 & 0  \\ 0 & 0  & 2 & 1  \\ 1 & 0 & 0  & 2 \\ 2&1&0&0  \end{array} \right] - \left[ \begin{array}{rrrrr} 0& 0 & 1 & 2  \\ 2 & 0  & 0 & 1  \\ 1 & 2 & 0  & 0 \\ 0&1&2&0  \end{array} \right]  =  \left[ \begin{array}{rrrrr} 0& 2 & 0 & -2  \\ -2 & 0  & 2 & 0  \\ 0 & -2 & 0  & 2 \\ 2&0&-2&0  \end{array} \right]
\end{equation*}

\begin{equation*}
D_0 = P_0-N_0=\left[ \begin{array}{rrrrr} 0& 2 & 1 & 2  \\ 2 & 0  & 2 & 1  \\ 1 & 2 & 0  & 2 \\ 2&1&2&0  \end{array} \right] - \left[ \begin{array}{rrrrr} 0& 1 & 2 & 1  \\ 1 & 0  & 1 & 2  \\ 2 & 1 & 0  & 1 \\ 1&2&1&0  \end{array} \right] =  \left[ \begin{array}{rrrrr} 0& 1 & -1 & 1  \\ 1 & 0  & 1& -1  \\ -1 & 1 & 0  & 1 \\ 1&-1&1&0  \end{array} \right] 
\end{equation*}

with the Frobenius companion matrix, 
\begin{equation*}
\left[ \begin{array}{rrrrrr} 
0 & I_4 & 0 & 0 \\  
0 & 0 & I_4 & 0   \\  
0 & 0 & 0 & I_4   \\  
-D_0 & -D_1 & -D_2 & -D_3  
\end{array} \right]
\end{equation*}
which through its spectrum provides an intrinsic stability measure of the weighted cycle ledgers. 

It is worth noting that the form of the weighted matrices so far described, whether the weights of the positive cycles, negative cycles, or their difference, are all circulant, that is, 
\begin{equation*}
\left[ \begin{array}{rrrr} 
\alpha & \beta & \gamma & \delta  \\  
\delta  & \alpha & \beta & \gamma  \\  
\gamma & \delta & \alpha & \beta  \\  
\beta & \gamma  & \delta  & \alpha  \\
\end{array} \right],
\end{equation*} 
which although common, is not a general property.

\section{The cycle spectrum}
More broadly, we can apply the approach described above to determine the intrinsic or cycle spectral signature for any sign-restricted (and zero element) system, from which properties associated solely with the structure, rather than any arbitrary values, can be determined, and which also differ from the spectrum of the original sign-restricted matrix from which the weighted cycle counts were taken. It is natural to consider the analogy with spectral graph theory ([7],[8]) and the explicit spectral relationship to a given form, merged with the way sign-restricted matrix forms are considered in term of their stability ([9],[10]).

As an example of how one can examine the individual coefficient matrix weights to consider the intrinsic stability of the system structure, consider the following $2 \times 2$ forms along with the dominant eigenvalue of their Frobenius companion matrix ($\mathcal{F}_{\mathbf{A}}$) constructed from the pencil of their D-tuples: \\

\vspace{1em}
\noindent
\begin{center}
\begin{minipage}[b]{0.22\textwidth}
\centering
\textbf{a}\\[0.4em]
$\begin{pmatrix}- & - \\ - & -\end{pmatrix}$\\[0.5em]
$\displaystyle \lambda_{\max}(\mathcal{F}_{\mathbf{A}}) = 0$
\end{minipage}\hfill
\begin{minipage}[b]{0.22\textwidth}
\centering
\textbf{b}\\[0.4em]
$\begin{pmatrix}- & + \\ - & -\end{pmatrix}$\\[0.5em]
$0$
\end{minipage}\hfill
\begin{minipage}[b]{0.22\textwidth}
\centering
\textbf{c}\\[0.4em]
$\begin{pmatrix}+ & - \\ - & -\end{pmatrix}$\\[0.5em]
$\sqrt{3}$
\end{minipage}\hfill
\begin{minipage}[b]{0.22\textwidth}
\centering
\textbf{d}\\[0.4em]
$\begin{pmatrix}- & 0 \\ - & -\end{pmatrix}$\\[0.5em]
$-\tfrac12\pm i\tfrac{\sqrt3}{2}$ 
\end{minipage}
\end{center}
\vspace{1em}

\vspace{0.3cm}
The associated matrix polynomial, with its ledger of counts,  ($\mathcal{L_{\mathbf{A}}}$), of form \textbf{a} is,\\ \\

\hspace{0.5cm}
\resizebox{0.8\hsize}{!}
{$\mathcal{L_{\mathbf{a}}}(x) = I_2x^2 + \left( \left[ \begin{array}{rr} 1&0  \\ 0 & 1   \end{array} \right] - \left[ \begin{array}{rr} 0& 0   \\ 0 & 0  \end{array} \right] \right) x +  \left( \left[ \begin{array}{rr} 1& 0 \\ 0 & 1   \end{array} \right] - \left[ \begin{array}{rr} 0& 1   \\ 1 & 0  \end{array} \right] \right)$} \\\\

which, through shifting the negative cycle weights in the $x_0$ coefficient over to the positive side, produces the matrix polynomial of form \textbf{b},  \\ \\ 

\hspace{0.5cm}
\resizebox{0.8\hsize}{!}
{$\mathcal{L_{\mathbf{b}}}(x) = I_2x^2 + \left( \left[ \begin{array}{rr} 1& 0 \\ 0 & 1   \end{array} \right] - \left[ \begin{array}{rr} 0& 0   \\ 0 & 0  \end{array} \right] \right) x +  \left( \left[ \begin{array}{rr} 1& 1 \\ 1 & 1   \end{array} \right] - \left[ \begin{array}{rr} 0& 0   \\ 0 & 0  \end{array} \right] \right)$}.\\\\

This has the effect of making the coefficient more positive and therefore the system more or equally stable (in this case equally). In form \textbf{c}, making the top left element positive instead of the top right, has a more dramatic effect, shifting all the elements in $x^0$ into negative cycles, as well as one of the elements in $x$, \\ \\  

\hspace{0.5cm}
\resizebox{0.8\hsize}{!}
{$\mathcal{L_{\mathbf{c}}}(x) = I_2x^2 + \left( \left[ \begin{array}{rr} 0& 0 \\ 0 & 1   \end{array} \right] - \left[ \begin{array}{rr} 1& 0   \\ 0 & 0  \end{array} \right] \right) x +  \left( \left[ \begin{array}{rr} 0& 0 \\ 0 & 0   \end{array} \right] - \left[ \begin{array}{rr} 1& 1   \\ 1 & 1  \end{array} \right] \right)$}
\\\\
which produces a much larger maximal eigenvalue. However, if we instead make the top right element zero as in form \textbf{d}, then not only are all the elements in the positive cycles, but we also minimise their number\\ \\

\hspace{0.5cm}
\resizebox{0.8\hsize}{!}
{$\mathcal{L_{\mathbf{d}}}(x) = I_2x^2 + \left( \left[ \begin{array}{rr} 1& 0 \\ 0 & 1   \end{array} \right] - \left[ \begin{array}{rr} 0& 0   \\ 0 & 0  \end{array} \right] \right) x +  \left( \left[ \begin{array}{rr} 1& 0 \\ 0 & 1   \end{array} \right] - \left[ \begin{array}{rr} 0& 0   \\ 0 & 0  \end{array} \right] \right)$}.\\ \\

\hspace{-0.5cm}
This system (\textbf{d}), along with its cospectral forms, 

\begin{equation*}    
\centering
\left[\begin{array}{rr} -& - \\ 0 & -  \end{array} \right]  \left[\begin{array}{rr} -& 0 \\ 0 & -   \end{array}\right] \left[\begin{array}{rr} -& + \\ 0 & -  \end{array} \right]
\left[\begin{array}{rr} -& 0 \\ + & -  \end{array} \right]
\end{equation*} \\

\hspace{-0.5cm}
with its dominant eigenvalue of $-\tfrac12\pm i\tfrac{\sqrt3}{2}$ is the smallest possible maximal eigenvalue that we find in the $3^{2 \times 2}$ $2 \times 2$ systems. Therefore, structurally we could consider these five systems, in the sense in which we have constructed them, intrinsically the most stable. Table 2 lists the 81 $2 \times 2$ possible matrices which collapse to 27 distinct ledger types (as indicated by the \# column), 16 cospectral groups of the original $\mathbf{A}$ matrices and 12 companion ($\mathcal{F}_{\mathbf{A}}$) cospectral groups.

\begingroup
\scriptsize
\begin{longtable}{@{}r c c c c c@{}}
\caption{\textit{The ledger group, sign restricted form, associated $D$-tuples, pencil spectrum, and spectrum of all $3^{4}=81$ matrices. $\varphi=(1+\sqrt{5})/2\approx 1.618$ denotes the golden ratio.}}\label{tab:2x2-mu}\\
\toprule
\# & $\mathbf{A}$ & $D_1$ & $D_0$ & $\lambda(\mathcal{F}_{\mathbf{A}})$ & $\lambda(\mathbf{A})$ \\ \midrule
\endfirsthead
\multicolumn{6}{c}{{\tablename\ \thetable{} -- continued}}\\
\toprule \# & $\mathbf{A}$ & $D_1$ & $D_0$ & $\lambda(\mathcal{F}_{\mathbf{A}})$  & $\lambda(\mathbf{A})$ \\ \midrule
\endhead
\midrule \multicolumn{6}{r}{}\\
\endfoot
\bottomrule
\endlastfoot
 1 & $\begin{pmatrix}- & - \\ 0 & -\end{pmatrix}$ & $\begin{pmatrix}1 & 0 \\ 0 & 1\end{pmatrix}$ $-$ $\begin{pmatrix}0 & 0 \\ 0 & 0\end{pmatrix}$ & $\begin{pmatrix}1 & 0 \\ 0 & 1\end{pmatrix}$ $-$ $\begin{pmatrix}0 & 0 \\ 0 & 0\end{pmatrix}$ & $\{\left(-\tfrac12 \pm \tfrac{\sqrt3}{2}\,\mathrm{i}\right)^{2}\}$ & $\{\left(-1\right)^{2}\}$ \\
 1 & $\begin{pmatrix}- & 0 \\ - & -\end{pmatrix}$ & $\begin{pmatrix}1 & 0 \\ 0 & 1\end{pmatrix}$ $-$ $\begin{pmatrix}0 & 0 \\ 0 & 0\end{pmatrix}$ & $\begin{pmatrix}1 & 0 \\ 0 & 1\end{pmatrix}$ $-$ $\begin{pmatrix}0 & 0 \\ 0 & 0\end{pmatrix}$ & $\{\left(-\tfrac12 \pm \tfrac{\sqrt3}{2}\,\mathrm{i}\right)^{2}\}$ & $\{\left(-1\right)^{2}\}$ \\
 1 & $\begin{pmatrix}- & 0 \\ 0 & -\end{pmatrix}$ & $\begin{pmatrix}1 & 0 \\ 0 & 1\end{pmatrix}$ $-$ $\begin{pmatrix}0 & 0 \\ 0 & 0\end{pmatrix}$ & $\begin{pmatrix}1 & 0 \\ 0 & 1\end{pmatrix}$ $-$ $\begin{pmatrix}0 & 0 \\ 0 & 0\end{pmatrix}$ & $\{\left(-\tfrac12 \pm \tfrac{\sqrt3}{2}\,\mathrm{i}\right)^{2}\}$ & $\{\left(-1\right)^{2}\}$ \\
 1 & $\begin{pmatrix}- & 0 \\ + & -\end{pmatrix}$ & $\begin{pmatrix}1 & 0 \\ 0 & 1\end{pmatrix}$ $-$ $\begin{pmatrix}0 & 0 \\ 0 & 0\end{pmatrix}$ & $\begin{pmatrix}1 & 0 \\ 0 & 1\end{pmatrix}$ $-$ $\begin{pmatrix}0 & 0 \\ 0 & 0\end{pmatrix}$ & $\{\left(-\tfrac12 \pm \tfrac{\sqrt3}{2}\,\mathrm{i}\right)^{2}\}$ & $\{\left(-1\right)^{2}\}$ \\
 1 & $\begin{pmatrix}- & + \\ 0 & -\end{pmatrix}$ & $\begin{pmatrix}1 & 0 \\ 0 & 1\end{pmatrix}$ $-$ $\begin{pmatrix}0 & 0 \\ 0 & 0\end{pmatrix}$ & $\begin{pmatrix}1 & 0 \\ 0 & 1\end{pmatrix}$ $-$ $\begin{pmatrix}0 & 0 \\ 0 & 0\end{pmatrix}$ & $\{\left(-\tfrac12 \pm \tfrac{\sqrt3}{2}\,\mathrm{i}\right)^{2}\}$ & $\{\left(-1\right)^{2}\}$ \\
 2 & $\begin{pmatrix}- & - \\ + & -\end{pmatrix}$ & $\begin{pmatrix}1 & 0 \\ 0 & 1\end{pmatrix}$ $-$ $\begin{pmatrix}0 & 0 \\ 0 & 0\end{pmatrix}$ & $\begin{pmatrix}1 & 1 \\ 1 & 1\end{pmatrix}$ $-$ $\begin{pmatrix}0 & 0 \\ 0 & 0\end{pmatrix}$ & $\{-1, -\tfrac12 \pm \tfrac{\sqrt7}{2}\,\mathrm{i}, 0\}$ & $\{-1 \pm \mathrm{i}\}$ \\
 2 & $\begin{pmatrix}- & + \\ - & -\end{pmatrix}$ & $\begin{pmatrix}1 & 0 \\ 0 & 1\end{pmatrix}$ $-$ $\begin{pmatrix}0 & 0 \\ 0 & 0\end{pmatrix}$ & $\begin{pmatrix}1 & 1 \\ 1 & 1\end{pmatrix}$ $-$ $\begin{pmatrix}0 & 0 \\ 0 & 0\end{pmatrix}$ & $\{-1, -\tfrac12 \pm \tfrac{\sqrt7}{2}\,\mathrm{i},0\}$ & $\{-1 \pm \mathrm{i}\}$ \\
 3 & $\begin{pmatrix}0 & - \\ 0 & -\end{pmatrix}$ & $\begin{pmatrix}0 & 0 \\ 0 & 1\end{pmatrix}$ $-$ $\begin{pmatrix}0 & 0 \\ 0 & 0\end{pmatrix}$ & $\begin{pmatrix}0 & 0 \\ 0 & 0\end{pmatrix}$ $-$ $\begin{pmatrix}0 & 0 \\ 0 & 0\end{pmatrix}$ & $\{-1, 0^{3}\}$ & $\{-1, 0\}$ \\
 3 & $\begin{pmatrix}0 & 0 \\ - & -\end{pmatrix}$ & $\begin{pmatrix}0 & 0 \\ 0 & 1\end{pmatrix}$ $-$ $\begin{pmatrix}0 & 0 \\ 0 & 0\end{pmatrix}$ & $\begin{pmatrix}0 & 0 \\ 0 & 0\end{pmatrix}$ $-$ $\begin{pmatrix}0 & 0 \\ 0 & 0\end{pmatrix}$ & $\{-1, 0^{3}\}$ & $\{-1, 0\}$ \\
 3 & $\begin{pmatrix}0 & 0 \\ 0 & -\end{pmatrix}$ & $\begin{pmatrix}0 & 0 \\ 0 & 1\end{pmatrix}$ $-$ $\begin{pmatrix}0 & 0 \\ 0 & 0\end{pmatrix}$ & $\begin{pmatrix}0 & 0 \\ 0 & 0\end{pmatrix}$ $-$ $\begin{pmatrix}0 & 0 \\ 0 & 0\end{pmatrix}$ & $\{-1, 0^{3}\}$ & $\{-1, 0\}$ \\
 3 & $\begin{pmatrix}0 & 0 \\ + & -\end{pmatrix}$ & $\begin{pmatrix}0 & 0 \\ 0 & 1\end{pmatrix}$ $-$ $\begin{pmatrix}0 & 0 \\ 0 & 0\end{pmatrix}$ & $\begin{pmatrix}0 & 0 \\ 0 & 0\end{pmatrix}$ $-$ $\begin{pmatrix}0 & 0 \\ 0 & 0\end{pmatrix}$ & $\{-1, 0^{3}\}$ & $\{-1, 0\}$ \\
 3 & $\begin{pmatrix}0 & + \\ 0 & -\end{pmatrix}$ & $\begin{pmatrix}0 & 0 \\ 0 & 1\end{pmatrix}$ $-$ $\begin{pmatrix}0 & 0 \\ 0 & 0\end{pmatrix}$ & $\begin{pmatrix}0 & 0 \\ 0 & 0\end{pmatrix}$ $-$ $\begin{pmatrix}0 & 0 \\ 0 & 0\end{pmatrix}$ & $\{-1, 0^{3}\}$ & $\{-1, 0\}$ \\
 4 & $\begin{pmatrix}- & - \\ 0 & 0\end{pmatrix}$ & $\begin{pmatrix}1 & 0 \\ 0 & 0\end{pmatrix}$ $-$ $\begin{pmatrix}0 & 0 \\ 0 & 0\end{pmatrix}$ & $\begin{pmatrix}0 & 0 \\ 0 & 0\end{pmatrix}$ $-$ $\begin{pmatrix}0 & 0 \\ 0 & 0\end{pmatrix}$ & $\{-1, 0^{3}\}$ & $\{-1, 0\}$ \\
 4 & $\begin{pmatrix}- & 0 \\ - & 0\end{pmatrix}$ & $\begin{pmatrix}1 & 0 \\ 0 & 0\end{pmatrix}$ $-$ $\begin{pmatrix}0 & 0 \\ 0 & 0\end{pmatrix}$ & $\begin{pmatrix}0 & 0 \\ 0 & 0\end{pmatrix}$ $-$ $\begin{pmatrix}0 & 0 \\ 0 & 0\end{pmatrix}$ & $\{-1, 0^{3}\}$ & $\{-1, 0\}$ \\
 4 & $\begin{pmatrix}- & 0 \\ 0 & 0\end{pmatrix}$ & $\begin{pmatrix}1 & 0 \\ 0 & 0\end{pmatrix}$ $-$ $\begin{pmatrix}0 & 0 \\ 0 & 0\end{pmatrix}$ & $\begin{pmatrix}0 & 0 \\ 0 & 0\end{pmatrix}$ $-$ $\begin{pmatrix}0 & 0 \\ 0 & 0\end{pmatrix}$ & $\{-1, 0^{3}\}$ & $\{-1, 0\}$ \\
 4 & $\begin{pmatrix}- & 0 \\ + & 0\end{pmatrix}$ & $\begin{pmatrix}1 & 0 \\ 0 & 0\end{pmatrix}$ $-$ $\begin{pmatrix}0 & 0 \\ 0 & 0\end{pmatrix}$ & $\begin{pmatrix}0 & 0 \\ 0 & 0\end{pmatrix}$ $-$ $\begin{pmatrix}0 & 0 \\ 0 & 0\end{pmatrix}$ & $\{-1, 0^{3}\}$ & $\{-1, 0\}$ \\
 4 & $\begin{pmatrix}- & + \\ 0 & 0\end{pmatrix}$ & $\begin{pmatrix}1 & 0 \\ 0 & 0\end{pmatrix}$ $-$ $\begin{pmatrix}0 & 0 \\ 0 & 0\end{pmatrix}$ & $\begin{pmatrix}0 & 0 \\ 0 & 0\end{pmatrix}$ $-$ $\begin{pmatrix}0 & 0 \\ 0 & 0\end{pmatrix}$ & $\{-1, 0^{3}\}$ & $\{-1, 0\}$ \\
 5 & $\begin{pmatrix}0 & - \\ 0 & 0\end{pmatrix}$ & $\begin{pmatrix}0 & 0 \\ 0 & 0\end{pmatrix}$ $-$ $\begin{pmatrix}0 & 0 \\ 0 & 0\end{pmatrix}$ & $\begin{pmatrix}0 & 0 \\ 0 & 0\end{pmatrix}$ $-$ $\begin{pmatrix}0 & 0 \\ 0 & 0\end{pmatrix}$ & $\{0^{4}\}$ & $\{0^{2}\}$ \\
 5 & $\begin{pmatrix}0 & 0 \\ - & 0\end{pmatrix}$ & $\begin{pmatrix}0 & 0 \\ 0 & 0\end{pmatrix}$ $-$ $\begin{pmatrix}0 & 0 \\ 0 & 0\end{pmatrix}$ & $\begin{pmatrix}0 & 0 \\ 0 & 0\end{pmatrix}$ $-$ $\begin{pmatrix}0 & 0 \\ 0 & 0\end{pmatrix}$ & $\{0^{4}\}$ & $\{0^{2}\}$ \\
 5 & $\begin{pmatrix}0 & 0 \\ 0 & 0\end{pmatrix}$ & $\begin{pmatrix}0 & 0 \\ 0 & 0\end{pmatrix}$ $-$ $\begin{pmatrix}0 & 0 \\ 0 & 0\end{pmatrix}$ & $\begin{pmatrix}0 & 0 \\ 0 & 0\end{pmatrix}$ $-$ $\begin{pmatrix}0 & 0 \\ 0 & 0\end{pmatrix}$ & $\{0^{4}\}$ & $\{0^{2}\}$ \\
 5 & $\begin{pmatrix}0 & 0 \\ + & 0\end{pmatrix}$ & $\begin{pmatrix}0 & 0 \\ 0 & 0\end{pmatrix}$ $-$ $\begin{pmatrix}0 & 0 \\ 0 & 0\end{pmatrix}$ & $\begin{pmatrix}0 & 0 \\ 0 & 0\end{pmatrix}$ $-$ $\begin{pmatrix}0 & 0 \\ 0 & 0\end{pmatrix}$ & $\{0^{4}\}$ & $\{0^{2}\}$ \\
 5 & $\begin{pmatrix}0 & + \\ 0 & 0\end{pmatrix}$ & $\begin{pmatrix}0 & 0 \\ 0 & 0\end{pmatrix}$ $-$ $\begin{pmatrix}0 & 0 \\ 0 & 0\end{pmatrix}$ & $\begin{pmatrix}0 & 0 \\ 0 & 0\end{pmatrix}$ $-$ $\begin{pmatrix}0 & 0 \\ 0 & 0\end{pmatrix}$ & $\{0^{4}\}$ & $\{0^{2}\}$ \\
 6 & $\begin{pmatrix}- & - \\ - & -\end{pmatrix}$ & $\begin{pmatrix}1 & 0 \\ 0 & 1\end{pmatrix}$ $-$ $\begin{pmatrix}0 & 0 \\ 0 & 0\end{pmatrix}$ & $\begin{pmatrix}1 & 0 \\ 0 & 1\end{pmatrix}$ $-$ $\begin{pmatrix}0 & 1 \\ 1 & 0\end{pmatrix}$ & $\{-1, -\tfrac12 \pm \tfrac{\sqrt7}{2}\,\mathrm{i},0\}$ & $\{-2, 0\}$ \\
 6 & $\begin{pmatrix}- & + \\ + & -\end{pmatrix}$ & $\begin{pmatrix}1 & 0 \\ 0 & 1\end{pmatrix}$ $-$ $\begin{pmatrix}0 & 0 \\ 0 & 0\end{pmatrix}$ & $\begin{pmatrix}1 & 0 \\ 0 & 1\end{pmatrix}$ $-$ $\begin{pmatrix}0 & 1 \\ 1 & 0\end{pmatrix}$ & $\{-1, -\tfrac12 \pm \tfrac{\sqrt7}{2}\,\mathrm{i},0\}$ & $\{-2, 0\}$ \\
 7 & $\begin{pmatrix}+ & - \\ 0 & +\end{pmatrix}$ & $\begin{pmatrix}0 & 0 \\ 0 & 0\end{pmatrix}$ $-$ $\begin{pmatrix}1 & 0 \\ 0 & 1\end{pmatrix}$ & $\begin{pmatrix}1 & 0 \\ 0 & 1\end{pmatrix}$ $-$ $\begin{pmatrix}0 & 0 \\ 0 & 0\end{pmatrix}$ & $\{\left(\tfrac12 \pm \tfrac{\sqrt3}{2}\,\mathrm{i}\right)^{2}\}$ & $\{1^{2}\}$ \\
 7 & $\begin{pmatrix}+ & 0 \\ - & +\end{pmatrix}$ & $\begin{pmatrix}0 & 0 \\ 0 & 0\end{pmatrix}$ $-$ $\begin{pmatrix}1 & 0 \\ 0 & 1\end{pmatrix}$ & $\begin{pmatrix}1 & 0 \\ 0 & 1\end{pmatrix}$ $-$ $\begin{pmatrix}0 & 0 \\ 0 & 0\end{pmatrix}$ & $\{\left(\tfrac12 \pm \tfrac{\sqrt3}{2}\,\mathrm{i}\right)^{2}\}$ & $\{1^{2}\}$ \\
 7 & $\begin{pmatrix}+ & 0 \\ 0 & +\end{pmatrix}$ & $\begin{pmatrix}0 & 0 \\ 0 & 0\end{pmatrix}$ $-$ $\begin{pmatrix}1 & 0 \\ 0 & 1\end{pmatrix}$ & $\begin{pmatrix}1 & 0 \\ 0 & 1\end{pmatrix}$ $-$ $\begin{pmatrix}0 & 0 \\ 0 & 0\end{pmatrix}$ & $\{\left(\tfrac12 \pm \tfrac{\sqrt3}{2}\,\mathrm{i}\right)^{2}\}$ & $\{1^{2}\}$ \\
 7 & $\begin{pmatrix}+ & 0 \\ + & +\end{pmatrix}$ & $\begin{pmatrix}0 & 0 \\ 0 & 0\end{pmatrix}$ $-$ $\begin{pmatrix}1 & 0 \\ 0 & 1\end{pmatrix}$ & $\begin{pmatrix}1 & 0 \\ 0 & 1\end{pmatrix}$ $-$ $\begin{pmatrix}0 & 0 \\ 0 & 0\end{pmatrix}$ & $\{\left(\tfrac12 \pm \tfrac{\sqrt3}{2}\,\mathrm{i}\right)^{2}\}$ & $\{1^{2}\}$ \\
 7 & $\begin{pmatrix}+ & + \\ 0 & +\end{pmatrix}$ & $\begin{pmatrix}0 & 0 \\ 0 & 0\end{pmatrix}$ $-$ $\begin{pmatrix}1 & 0 \\ 0 & 1\end{pmatrix}$ & $\begin{pmatrix}1 & 0 \\ 0 & 1\end{pmatrix}$ $-$ $\begin{pmatrix}0 & 0 \\ 0 & 0\end{pmatrix}$ & $\{\left(\tfrac12 \pm \tfrac{\sqrt3}{2}\,\mathrm{i}\right)^{2}\}$ & $\{1^{2}\}$ \\
 8 & $\begin{pmatrix}0 & - \\ + & -\end{pmatrix}$ & $\begin{pmatrix}0 & 0 \\ 0 & 1\end{pmatrix}$ $-$ $\begin{pmatrix}0 & 0 \\ 0 & 0\end{pmatrix}$ & $\begin{pmatrix}0 & 1 \\ 1 & 0\end{pmatrix}$ $-$ $\begin{pmatrix}0 & 0 \\ 0 & 0\end{pmatrix}$ & $\{-1.38, -0.219 \pm 0.914\,\mathrm{i},0.819\}$ & $\{-\tfrac12 \pm \tfrac{\sqrt3}{2}\,\mathrm{i}\}$ \\
 8 & $\begin{pmatrix}0 & + \\ - & -\end{pmatrix}$ & $\begin{pmatrix}0 & 0 \\ 0 & 1\end{pmatrix}$ $-$ $\begin{pmatrix}0 & 0 \\ 0 & 0\end{pmatrix}$ & $\begin{pmatrix}0 & 1 \\ 1 & 0\end{pmatrix}$ $-$ $\begin{pmatrix}0 & 0 \\ 0 & 0\end{pmatrix}$ & $\{-1.38, -0.219 \pm 0.914\,\mathrm{i}, 0.819\}$ & $\{-\tfrac12 \pm \tfrac{\sqrt3}{2}\,\mathrm{i}\}$ \\
 9 & $\begin{pmatrix}- & - \\ + & 0\end{pmatrix}$ & $\begin{pmatrix}1 & 0 \\ 0 & 0\end{pmatrix}$ $-$ $\begin{pmatrix}0 & 0 \\ 0 & 0\end{pmatrix}$ & $\begin{pmatrix}0 & 1 \\ 1 & 0\end{pmatrix}$ $-$ $\begin{pmatrix}0 & 0 \\ 0 & 0\end{pmatrix}$ & $\{-1.38, -0.219 \pm 0.914\,\mathrm{i}, 0.819\}$ & $\{-\tfrac12 \pm \tfrac{\sqrt3}{2}\,\mathrm{i}\}$ \\
 9 & $\begin{pmatrix}- & + \\ - & 0\end{pmatrix}$ & $\begin{pmatrix}1 & 0 \\ 0 & 0\end{pmatrix}$ $-$ $\begin{pmatrix}0 & 0 \\ 0 & 0\end{pmatrix}$ & $\begin{pmatrix}0 & 1 \\ 1 & 0\end{pmatrix}$ $-$ $\begin{pmatrix}0 & 0 \\ 0 & 0\end{pmatrix}$ & $\{-1.38, -0.219 \pm 0.914\,\mathrm{i}, 0.819\}$ & $\{-\tfrac12 \pm \tfrac{\sqrt3}{2}\,\mathrm{i}\}$ \\
 10 & $\begin{pmatrix}0 & - \\ - & -\end{pmatrix}$ & $\begin{pmatrix}0 & 0 \\ 0 & 1\end{pmatrix}$ $-$ $\begin{pmatrix}0 & 0 \\ 0 & 0\end{pmatrix}$ & $\begin{pmatrix}0 & 0 \\ 0 & 0\end{pmatrix}$ $-$ $\begin{pmatrix}0 & 1 \\ 1 & 0\end{pmatrix}$ & $\{-1.38, -0.219 \pm 0.914\,\mathrm{i}, 0.819\}$ & $\{-\varphi, \varphi^{-1}\}$ \\
 10 & $\begin{pmatrix}0 & + \\ + & -\end{pmatrix}$ & $\begin{pmatrix}0 & 0 \\ 0 & 1\end{pmatrix}$ $-$ $\begin{pmatrix}0 & 0 \\ 0 & 0\end{pmatrix}$ & $\begin{pmatrix}0 & 0 \\ 0 & 0\end{pmatrix}$ $-$ $\begin{pmatrix}0 & 1 \\ 1 & 0\end{pmatrix}$ & $\{-1.38, -0.219 \pm 0.914\,\mathrm{i}, 0.819\}$ & $\{-\varphi, \varphi^{-1}\}$ \\
 11 & $\begin{pmatrix}- & - \\ - & 0\end{pmatrix}$ & $\begin{pmatrix}1 & 0 \\ 0 & 0\end{pmatrix}$ $-$ $\begin{pmatrix}0 & 0 \\ 0 & 0\end{pmatrix}$ & $\begin{pmatrix}0 & 0 \\ 0 & 0\end{pmatrix}$ $-$ $\begin{pmatrix}0 & 1 \\ 1 & 0\end{pmatrix}$ & $\{-1.38, -0.219 \pm 0.914\,\mathrm{i}, 0.819\}$ & $\{-\varphi, \varphi^{-1}\}$ \\
 11 & $\begin{pmatrix}- & + \\ + & 0\end{pmatrix}$ & $\begin{pmatrix}1 & 0 \\ 0 & 0\end{pmatrix}$ $-$ $\begin{pmatrix}0 & 0 \\ 0 & 0\end{pmatrix}$ & $\begin{pmatrix}0 & 0 \\ 0 & 0\end{pmatrix}$ $-$ $\begin{pmatrix}0 & 1 \\ 1 & 0\end{pmatrix}$ & $\{-1.38, -0.219 \pm 0.914\,\mathrm{i}, 0.819\}$ & $\{-\varphi, \varphi^{-1}\}$ \\
 12 & $\begin{pmatrix}0 & - \\ + & 0\end{pmatrix}$ & $\begin{pmatrix}0 & 0 \\ 0 & 0\end{pmatrix}$ $-$ $\begin{pmatrix}0 & 0 \\ 0 & 0\end{pmatrix}$ & $\begin{pmatrix}0 & 1 \\ 1 & 0\end{pmatrix}$ $-$ $\begin{pmatrix}0 & 0 \\ 0 & 0\end{pmatrix}$ & $\{-1, \pm \mathrm{i}, 1\}$ & $\{ \pm \mathrm{i}\}$ \\
 12 & $\begin{pmatrix}0 & + \\ - & 0\end{pmatrix}$ & $\begin{pmatrix}0 & 0 \\ 0 & 0\end{pmatrix}$ $-$ $\begin{pmatrix}0 & 0 \\ 0 & 0\end{pmatrix}$ & $\begin{pmatrix}0 & 1 \\ 1 & 0\end{pmatrix}$ $-$ $\begin{pmatrix}0 & 0 \\ 0 & 0\end{pmatrix}$ & $\{-1, \pm \mathrm{i}, 1\}$ & $\{\pm \mathrm{i}\}$ \\
 13 & $\begin{pmatrix}0 & - \\ - & 0\end{pmatrix}$ & $\begin{pmatrix}0 & 0 \\ 0 & 0\end{pmatrix}$ $-$ $\begin{pmatrix}0 & 0 \\ 0 & 0\end{pmatrix}$ & $\begin{pmatrix}0 & 0 \\ 0 & 0\end{pmatrix}$ $-$ $\begin{pmatrix}0 & 1 \\ 1 & 0\end{pmatrix}$ & $\{-1, \pm \mathrm{i}, 1\}$ & $\{-1, 1\}$ \\
 13 & $\begin{pmatrix}0 & + \\ + & 0\end{pmatrix}$ & $\begin{pmatrix}0 & 0 \\ 0 & 0\end{pmatrix}$ $-$ $\begin{pmatrix}0 & 0 \\ 0 & 0\end{pmatrix}$ & $\begin{pmatrix}0 & 0 \\ 0 & 0\end{pmatrix}$ $-$ $\begin{pmatrix}0 & 1 \\ 1 & 0\end{pmatrix}$ & $\{-1, \pm \mathrm{i}, 1\}$ & $\{-1, 1\}$ \\
 14 & $\begin{pmatrix}+ & - \\ 0 & 0\end{pmatrix}$ & $\begin{pmatrix}0 & 0 \\ 0 & 0\end{pmatrix}$ $-$ $\begin{pmatrix}1 & 0 \\ 0 & 0\end{pmatrix}$ & $\begin{pmatrix}0 & 0 \\ 0 & 0\end{pmatrix}$ $-$ $\begin{pmatrix}0 & 0 \\ 0 & 0\end{pmatrix}$ & $\{0^{3}, 1\}$ & $\{0, 1\}$ \\
 14 & $\begin{pmatrix}+ & 0 \\ - & 0\end{pmatrix}$ & $\begin{pmatrix}0 & 0 \\ 0 & 0\end{pmatrix}$ $-$ $\begin{pmatrix}1 & 0 \\ 0 & 0\end{pmatrix}$ & $\begin{pmatrix}0 & 0 \\ 0 & 0\end{pmatrix}$ $-$ $\begin{pmatrix}0 & 0 \\ 0 & 0\end{pmatrix}$ & $\{0^{3}, 1\}$ & $\{0, 1\}$ \\
 14 & $\begin{pmatrix}+ & 0 \\ 0 & 0\end{pmatrix}$ & $\begin{pmatrix}0 & 0 \\ 0 & 0\end{pmatrix}$ $-$ $\begin{pmatrix}1 & 0 \\ 0 & 0\end{pmatrix}$ & $\begin{pmatrix}0 & 0 \\ 0 & 0\end{pmatrix}$ $-$ $\begin{pmatrix}0 & 0 \\ 0 & 0\end{pmatrix}$ & $\{0^{3}, 1\}$ & $\{0, 1\}$ \\
 14 & $\begin{pmatrix}+ & 0 \\ + & 0\end{pmatrix}$ & $\begin{pmatrix}0 & 0 \\ 0 & 0\end{pmatrix}$ $-$ $\begin{pmatrix}1 & 0 \\ 0 & 0\end{pmatrix}$ & $\begin{pmatrix}0 & 0 \\ 0 & 0\end{pmatrix}$ $-$ $\begin{pmatrix}0 & 0 \\ 0 & 0\end{pmatrix}$ & $\{0^{3}, 1\}$ & $\{0, 1\}$ \\
 14 & $\begin{pmatrix}+ & + \\ 0 & 0\end{pmatrix}$ & $\begin{pmatrix}0 & 0 \\ 0 & 0\end{pmatrix}$ $-$ $\begin{pmatrix}1 & 0 \\ 0 & 0\end{pmatrix}$ & $\begin{pmatrix}0 & 0 \\ 0 & 0\end{pmatrix}$ $-$ $\begin{pmatrix}0 & 0 \\ 0 & 0\end{pmatrix}$ & $\{0^{3}, 1\}$ & $\{0, 1\}$ \\
 15 & $\begin{pmatrix}0 & - \\ 0 & +\end{pmatrix}$ & $\begin{pmatrix}0 & 0 \\ 0 & 0\end{pmatrix}$ $-$ $\begin{pmatrix}0 & 0 \\ 0 & 1\end{pmatrix}$ & $\begin{pmatrix}0 & 0 \\ 0 & 0\end{pmatrix}$ $-$ $\begin{pmatrix}0 & 0 \\ 0 & 0\end{pmatrix}$ & $\{0^{3}, 1\}$ & $\{0, 1\}$ \\
 15 & $\begin{pmatrix}0 & 0 \\ - & +\end{pmatrix}$ & $\begin{pmatrix}0 & 0 \\ 0 & 0\end{pmatrix}$ $-$ $\begin{pmatrix}0 & 0 \\ 0 & 1\end{pmatrix}$ & $\begin{pmatrix}0 & 0 \\ 0 & 0\end{pmatrix}$ $-$ $\begin{pmatrix}0 & 0 \\ 0 & 0\end{pmatrix}$ & $\{0^{3}, 1\}$ & $\{0, 1\}$ \\
 15 & $\begin{pmatrix}0 & 0 \\ 0 & +\end{pmatrix}$ & $\begin{pmatrix}0 & 0 \\ 0 & 0\end{pmatrix}$ $-$ $\begin{pmatrix}0 & 0 \\ 0 & 1\end{pmatrix}$ & $\begin{pmatrix}0 & 0 \\ 0 & 0\end{pmatrix}$ $-$ $\begin{pmatrix}0 & 0 \\ 0 & 0\end{pmatrix}$ & $\{0^{3}, 1\}$ & $\{0, 1\}$ \\
 15 & $\begin{pmatrix}0 & 0 \\ + & +\end{pmatrix}$ & $\begin{pmatrix}0 & 0 \\ 0 & 0\end{pmatrix}$ $-$ $\begin{pmatrix}0 & 0 \\ 0 & 1\end{pmatrix}$ & $\begin{pmatrix}0 & 0 \\ 0 & 0\end{pmatrix}$ $-$ $\begin{pmatrix}0 & 0 \\ 0 & 0\end{pmatrix}$ & $\{0^{3}, 1\}$ & $\{0, 1\}$ \\
 15 & $\begin{pmatrix}0 & + \\ 0 & +\end{pmatrix}$ & $\begin{pmatrix}0 & 0 \\ 0 & 0\end{pmatrix}$ $-$ $\begin{pmatrix}0 & 0 \\ 0 & 1\end{pmatrix}$ & $\begin{pmatrix}0 & 0 \\ 0 & 0\end{pmatrix}$ $-$ $\begin{pmatrix}0 & 0 \\ 0 & 0\end{pmatrix}$ & $\{0^{3}, 1\}$ & $\{0, 1\}$ \\
 16 & $\begin{pmatrix}+ & - \\ + & +\end{pmatrix}$ & $\begin{pmatrix}0 & 0 \\ 0 & 0\end{pmatrix}$ $-$ $\begin{pmatrix}1 & 0 \\ 0 & 1\end{pmatrix}$ & $\begin{pmatrix}1 & 1 \\ 1 & 1\end{pmatrix}$ $-$ $\begin{pmatrix}0 & 0 \\ 0 & 0\end{pmatrix}$ & $\{0, \tfrac12 \pm \tfrac{\sqrt7}{2}\,\mathrm{i}, 1\}$ & $\{1 \pm \mathrm{i}\}$ \\
 16 & $\begin{pmatrix}+ & + \\ - & +\end{pmatrix}$ & $\begin{pmatrix}0 & 0 \\ 0 & 0\end{pmatrix}$ $-$ $\begin{pmatrix}1 & 0 \\ 0 & 1\end{pmatrix}$ & $\begin{pmatrix}1 & 1 \\ 1 & 1\end{pmatrix}$ $-$ $\begin{pmatrix}0 & 0 \\ 0 & 0\end{pmatrix}$ & $\{0, \tfrac12 \pm \tfrac{\sqrt7}{2}\,\mathrm{i}, 1\}$ & $\{1 \pm \mathrm{i}\}$ \\
 17 & $\begin{pmatrix}+ & - \\ - & +\end{pmatrix}$ & $\begin{pmatrix}0 & 0 \\ 0 & 0\end{pmatrix}$ $-$ $\begin{pmatrix}1 & 0 \\ 0 & 1\end{pmatrix}$ & $\begin{pmatrix}1 & 0 \\ 0 & 1\end{pmatrix}$ $-$ $\begin{pmatrix}0 & 1 \\ 1 & 0\end{pmatrix}$ & $\{0, \tfrac12 \pm \tfrac{\sqrt7}{2}\,\mathrm{i}, 1\}$ & $\{0, 2\}$ \\
 17 & $\begin{pmatrix}+ & + \\ + & +\end{pmatrix}$ & $\begin{pmatrix}0 & 0 \\ 0 & 0\end{pmatrix}$ $-$ $\begin{pmatrix}1 & 0 \\ 0 & 1\end{pmatrix}$ & $\begin{pmatrix}1 & 0 \\ 0 & 1\end{pmatrix}$ $-$ $\begin{pmatrix}0 & 1 \\ 1 & 0\end{pmatrix}$ & $\{0, \tfrac12 \pm \tfrac{\sqrt7}{2}\,\mathrm{i}, 1\}$ & $\{0, 2\}$ \\
 18 & $\begin{pmatrix}+ & - \\ + & 0\end{pmatrix}$ & $\begin{pmatrix}0 & 0 \\ 0 & 0\end{pmatrix}$ $-$ $\begin{pmatrix}1 & 0 \\ 0 & 0\end{pmatrix}$ & $\begin{pmatrix}0 & 1 \\ 1 & 0\end{pmatrix}$ $-$ $\begin{pmatrix}0 & 0 \\ 0 & 0\end{pmatrix}$ & $\{-0.819, 0.219 \pm 0.914\,\mathrm{i}, 1.38\}$ & $\{\tfrac12 \pm \tfrac{\sqrt3}{2}\,\mathrm{i}\}$ \\
 18 & $\begin{pmatrix}+ & + \\ - & 0\end{pmatrix}$ & $\begin{pmatrix}0 & 0 \\ 0 & 0\end{pmatrix}$ $-$ $\begin{pmatrix}1 & 0 \\ 0 & 0\end{pmatrix}$ & $\begin{pmatrix}0 & 1 \\ 1 & 0\end{pmatrix}$ $-$ $\begin{pmatrix}0 & 0 \\ 0 & 0\end{pmatrix}$ & $\{-0.819, 0.219 \pm 0.914\,\mathrm{i}, 1.38\}$ & $\{\tfrac12 \pm \tfrac{\sqrt3}{2}\,\mathrm{i}\}$ \\
 19 & $\begin{pmatrix}0 & - \\ + & +\end{pmatrix}$ & $\begin{pmatrix}0 & 0 \\ 0 & 0\end{pmatrix}$ $-$ $\begin{pmatrix}0 & 0 \\ 0 & 1\end{pmatrix}$ & $\begin{pmatrix}0 & 1 \\ 1 & 0\end{pmatrix}$ $-$ $\begin{pmatrix}0 & 0 \\ 0 & 0\end{pmatrix}$ & $\{-0.819, 0.219 \pm 0.914\,\mathrm{i}, 1.38\}$ & $\{\tfrac12 \pm \tfrac{\sqrt3}{2}\,\mathrm{i}\}$ \\
 19 & $\begin{pmatrix}0 & + \\ - & +\end{pmatrix}$ & $\begin{pmatrix}0 & 0 \\ 0 & 0\end{pmatrix}$ $-$ $\begin{pmatrix}0 & 0 \\ 0 & 1\end{pmatrix}$ & $\begin{pmatrix}0 & 1 \\ 1 & 0\end{pmatrix}$ $-$ $\begin{pmatrix}0 & 0 \\ 0 & 0\end{pmatrix}$ & $\{-0.819, 0.219 \pm 0.914\,\mathrm{i}, 1.38\}$ & $\{\tfrac12 \pm \tfrac{\sqrt3}{2}\,\mathrm{i}\}$ \\
 20 & $\begin{pmatrix}+ & - \\ - & 0\end{pmatrix}$ & $\begin{pmatrix}0 & 0 \\ 0 & 0\end{pmatrix}$ $-$ $\begin{pmatrix}1 & 0 \\ 0 & 0\end{pmatrix}$ & $\begin{pmatrix}0 & 0 \\ 0 & 0\end{pmatrix}$ $-$ $\begin{pmatrix}0 & 1 \\ 1 & 0\end{pmatrix}$ & $\{-0.819, 0.219 \pm 0.914\,\mathrm{i}, 1.38\}$ & $\{-\varphi^{-1}, \varphi\}$ \\
 20 & $\begin{pmatrix}+ & + \\ + & 0\end{pmatrix}$ & $\begin{pmatrix}0 & 0 \\ 0 & 0\end{pmatrix}$ $-$ $\begin{pmatrix}1 & 0 \\ 0 & 0\end{pmatrix}$ & $\begin{pmatrix}0 & 0 \\ 0 & 0\end{pmatrix}$ $-$ $\begin{pmatrix}0 & 1 \\ 1 & 0\end{pmatrix}$ & $\{-0.819, 0.219 \pm 0.914\,\mathrm{i}, 1.38\}$ & $\{-\varphi^{-1}, \varphi\}$ \\
 21 & $\begin{pmatrix}0 & - \\ - & +\end{pmatrix}$ & $\begin{pmatrix}0 & 0 \\ 0 & 0\end{pmatrix}$ $-$ $\begin{pmatrix}0 & 0 \\ 0 & 1\end{pmatrix}$ & $\begin{pmatrix}0 & 0 \\ 0 & 0\end{pmatrix}$ $-$ $\begin{pmatrix}0 & 1 \\ 1 & 0\end{pmatrix}$ & $\{-0.819, 0.219 \pm 0.914\,\mathrm{i}, 1.38\}$ & $\{-\varphi^{-1}, \varphi\}$ \\
 21 & $\begin{pmatrix}0 & + \\ + & +\end{pmatrix}$ & $\begin{pmatrix}0 & 0 \\ 0 & 0\end{pmatrix}$ $-$ $\begin{pmatrix}0 & 0 \\ 0 & 1\end{pmatrix}$ & $\begin{pmatrix}0 & 0 \\ 0 & 0\end{pmatrix}$ $-$ $\begin{pmatrix}0 & 1 \\ 1 & 0\end{pmatrix}$ & $\{-0.819, 0.219 \pm 0.914\,\mathrm{i}, 1.38\}$ & $\{-\varphi^{-1}, \varphi\}$ \\
 22 & $\begin{pmatrix}+ & - \\ 0 & -\end{pmatrix}$ & $\begin{pmatrix}0 & 0 \\ 0 & 1\end{pmatrix}$ $-$ $\begin{pmatrix}1 & 0 \\ 0 & 0\end{pmatrix}$ & $\begin{pmatrix}0 & 0 \\ 0 & 0\end{pmatrix}$ $-$ $\begin{pmatrix}1 & 0 \\ 0 & 1\end{pmatrix}$ & $\{-\varphi, -\varphi^{-1}, \varphi^{-1}, \varphi\}$ & $\{-1, 1\}$ \\
 22 & $\begin{pmatrix}+ & 0 \\ - & -\end{pmatrix}$ & $\begin{pmatrix}0 & 0 \\ 0 & 1\end{pmatrix}$ $-$ $\begin{pmatrix}1 & 0 \\ 0 & 0\end{pmatrix}$ & $\begin{pmatrix}0 & 0 \\ 0 & 0\end{pmatrix}$ $-$ $\begin{pmatrix}1 & 0 \\ 0 & 1\end{pmatrix}$ & $\{-\varphi, -\varphi^{-1}, \varphi^{-1}, \varphi\}$ & $\{-1, 1\}$ \\
 22 & $\begin{pmatrix}+ & 0 \\ 0 & -\end{pmatrix}$ & $\begin{pmatrix}0 & 0 \\ 0 & 1\end{pmatrix}$ $-$ $\begin{pmatrix}1 & 0 \\ 0 & 0\end{pmatrix}$ & $\begin{pmatrix}0 & 0 \\ 0 & 0\end{pmatrix}$ $-$ $\begin{pmatrix}1 & 0 \\ 0 & 1\end{pmatrix}$ & $\{-\varphi, -\varphi^{-1}, \varphi^{-1}, \varphi\}$ & $\{-1, 1\}$ \\
 22 & $\begin{pmatrix}+ & 0 \\ + & -\end{pmatrix}$ & $\begin{pmatrix}0 & 0 \\ 0 & 1\end{pmatrix}$ $-$ $\begin{pmatrix}1 & 0 \\ 0 & 0\end{pmatrix}$ & $\begin{pmatrix}0 & 0 \\ 0 & 0\end{pmatrix}$ $-$ $\begin{pmatrix}1 & 0 \\ 0 & 1\end{pmatrix}$ & $\{-\varphi, -\varphi^{-1}, \varphi^{-1}, \varphi\}$ & $\{-1, 1\}$ \\
 22 & $\begin{pmatrix}+ & + \\ 0 & -\end{pmatrix}$ & $\begin{pmatrix}0 & 0 \\ 0 & 1\end{pmatrix}$ $-$ $\begin{pmatrix}1 & 0 \\ 0 & 0\end{pmatrix}$ & $\begin{pmatrix}0 & 0 \\ 0 & 0\end{pmatrix}$ $-$ $\begin{pmatrix}1 & 0 \\ 0 & 1\end{pmatrix}$ & $\{-\varphi, -\varphi^{-1}, \varphi^{-1}, \varphi\}$ & $\{-1, 1\}$ \\
 23 & $\begin{pmatrix}- & - \\ 0 & +\end{pmatrix}$ & $\begin{pmatrix}1 & 0 \\ 0 & 0\end{pmatrix}$ $-$ $\begin{pmatrix}0 & 0 \\ 0 & 1\end{pmatrix}$ & $\begin{pmatrix}0 & 0 \\ 0 & 0\end{pmatrix}$ $-$ $\begin{pmatrix}1 & 0 \\ 0 & 1\end{pmatrix}$ & $\{-\varphi, -\varphi^{-1}, \varphi^{-1}, \varphi\}$ & $\{-1, 1\}$ \\
 23 & $\begin{pmatrix}- & 0 \\ - & +\end{pmatrix}$ & $\begin{pmatrix}1 & 0 \\ 0 & 0\end{pmatrix}$ $-$ $\begin{pmatrix}0 & 0 \\ 0 & 1\end{pmatrix}$ & $\begin{pmatrix}0 & 0 \\ 0 & 0\end{pmatrix}$ $-$ $\begin{pmatrix}1 & 0 \\ 0 & 1\end{pmatrix}$ & $\{-\varphi, -\varphi^{-1}, \varphi^{-1}, \varphi\}$ & $\{-1, 1\}$ \\
 23 & $\begin{pmatrix}- & 0 \\ 0 & +\end{pmatrix}$ & $\begin{pmatrix}1 & 0 \\ 0 & 0\end{pmatrix}$ $-$ $\begin{pmatrix}0 & 0 \\ 0 & 1\end{pmatrix}$ & $\begin{pmatrix}0 & 0 \\ 0 & 0\end{pmatrix}$ $-$ $\begin{pmatrix}1 & 0 \\ 0 & 1\end{pmatrix}$ & $\{-\varphi, -\varphi^{-1}, \varphi^{-1}, \varphi\}$ & $\{-1, 1\}$ \\
 23 & $\begin{pmatrix}- & 0 \\ + & +\end{pmatrix}$ & $\begin{pmatrix}1 & 0 \\ 0 & 0\end{pmatrix}$ $-$ $\begin{pmatrix}0 & 0 \\ 0 & 1\end{pmatrix}$ & $\begin{pmatrix}0 & 0 \\ 0 & 0\end{pmatrix}$ $-$ $\begin{pmatrix}1 & 0 \\ 0 & 1\end{pmatrix}$ & $\{-\varphi, -\varphi^{-1}, \varphi^{-1}, \varphi\}$ & $\{-1, 1\}$ \\
 23 & $\begin{pmatrix}- & + \\ 0 & +\end{pmatrix}$ & $\begin{pmatrix}1 & 0 \\ 0 & 0\end{pmatrix}$ $-$ $\begin{pmatrix}0 & 0 \\ 0 & 1\end{pmatrix}$ & $\begin{pmatrix}0 & 0 \\ 0 & 0\end{pmatrix}$ $-$ $\begin{pmatrix}1 & 0 \\ 0 & 1\end{pmatrix}$ & $\{-\varphi, -\varphi^{-1}, \varphi^{-1}, \varphi\}$ & $\{-1, 1\}$ \\
 24 & $\begin{pmatrix}+ & - \\ + & -\end{pmatrix}$ & $\begin{pmatrix}0 & 0 \\ 0 & 1\end{pmatrix}$ $-$ $\begin{pmatrix}1 & 0 \\ 0 & 0\end{pmatrix}$ & $\begin{pmatrix}0 & 1 \\ 1 & 0\end{pmatrix}$ $-$ $\begin{pmatrix}1 & 0 \\ 0 & 1\end{pmatrix}$ & $\{-\sqrt{3}, 0^{2}, \sqrt{3}\}$ & $\{0^{2}\}$ \\
 24 & $\begin{pmatrix}+ & + \\ - & -\end{pmatrix}$ & $\begin{pmatrix}0 & 0 \\ 0 & 1\end{pmatrix}$ $-$ $\begin{pmatrix}1 & 0 \\ 0 & 0\end{pmatrix}$ & $\begin{pmatrix}0 & 1 \\ 1 & 0\end{pmatrix}$ $-$ $\begin{pmatrix}1 & 0 \\ 0 & 1\end{pmatrix}$ & $\{-\sqrt{3}, 0^{2}, \sqrt{3}\}$ & $\{0^{2}\}$ \\
 25 & $\begin{pmatrix}- & - \\ + & +\end{pmatrix}$ & $\begin{pmatrix}1 & 0 \\ 0 & 0\end{pmatrix}$ $-$ $\begin{pmatrix}0 & 0 \\ 0 & 1\end{pmatrix}$ & $\begin{pmatrix}0 & 1 \\ 1 & 0\end{pmatrix}$ $-$ $\begin{pmatrix}1 & 0 \\ 0 & 1\end{pmatrix}$ & $\{-\sqrt{3}, 0^{2}, \sqrt{3}\}$ & $\{0^{2}\}$ \\
 25 & $\begin{pmatrix}- & + \\ - & +\end{pmatrix}$ & $\begin{pmatrix}1 & 0 \\ 0 & 0\end{pmatrix}$ $-$ $\begin{pmatrix}0 & 0 \\ 0 & 1\end{pmatrix}$ & $\begin{pmatrix}0 & 1 \\ 1 & 0\end{pmatrix}$ $-$ $\begin{pmatrix}1 & 0 \\ 0 & 1\end{pmatrix}$ & $\{-\sqrt{3}, 0^{2}, \sqrt{3}\}$ & $\{0^{2}\}$ \\
 26 & $\begin{pmatrix}+ & - \\ - & -\end{pmatrix}$ & $\begin{pmatrix}0 & 0 \\ 0 & 1\end{pmatrix}$ $-$ $\begin{pmatrix}1 & 0 \\ 0 & 0\end{pmatrix}$ & $\begin{pmatrix}0 & 0 \\ 0 & 0\end{pmatrix}$ $-$ $\begin{pmatrix}1 & 1 \\ 1 & 1\end{pmatrix}$ & $\{-\sqrt{3}, 0^{2}, \sqrt{3}\}$ & $\{-\sqrt{2}, \sqrt{2}\}$ \\
 26 & $\begin{pmatrix}+ & + \\ + & -\end{pmatrix}$ & $\begin{pmatrix}0 & 0 \\ 0 & 1\end{pmatrix}$ $-$ $\begin{pmatrix}1 & 0 \\ 0 & 0\end{pmatrix}$ & $\begin{pmatrix}0 & 0 \\ 0 & 0\end{pmatrix}$ $-$ $\begin{pmatrix}1 & 1 \\ 1 & 1\end{pmatrix}$ & $\{-\sqrt{3}, 0^{2}, \sqrt{3}\}$ & $\{-\sqrt{2}, \sqrt{2}\}$ \\
 27 & $\begin{pmatrix}- & - \\ - & +\end{pmatrix}$ & $\begin{pmatrix}1 & 0 \\ 0 & 0\end{pmatrix}$ $-$ $\begin{pmatrix}0 & 0 \\ 0 & 1\end{pmatrix}$ & $\begin{pmatrix}0 & 0 \\ 0 & 0\end{pmatrix}$ $-$ $\begin{pmatrix}1 & 1 \\ 1 & 1\end{pmatrix}$ & $\{-\sqrt{3}, 0^{2}, \sqrt{3}\}$ & $\{-\sqrt{2}, \sqrt{2}\}$ \\
 27 & $\begin{pmatrix}- & + \\ + & +\end{pmatrix}$ & $\begin{pmatrix}1 & 0 \\ 0 & 0\end{pmatrix}$ $-$ $\begin{pmatrix}0 & 0 \\ 0 & 1\end{pmatrix}$ & $\begin{pmatrix}0 & 0 \\ 0 & 0\end{pmatrix}$ $-$ $\begin{pmatrix}1 & 1 \\ 1 & 1\end{pmatrix}$ & $\{-\sqrt{3}, 0^{2}, \sqrt{3}\}$ & $\{-\sqrt{2}, \sqrt{2}\}$ \\
\end{longtable}
\endgroup

A question arises as to what relation may exist between the  original matrices $\mathbf{A}$ and the companion matrices of the pencils, $\mathcal{F}_{\mathbf{A}}$. We rewrite as a summary the dominant eigenvalues from Table 2 in Table 3.

\begin{table}[ht]
  \centering
  \caption{\textit{The set of dominant eigenvalues for all the $2\times2$ sign patterns for $\mathbf{A}$ and the
    companion of $\mathcal{L}_{\mathbf{A}}$, $\mathcal{F}_{\mathbf{A}}$, with counts for the number of cospectral groups, number of sign-restricted forms and distinct ledgers associated with each dominant eigenvalue. $\lambda_{\max}(\mathbf{A})\in\lambda(\mathbf{A})$ and $\lambda_{\max}(\mathcal{F}_\mathbf{A})\in\lambda(\mathcal{F}_\mathbf{A})$ shall be used to denote the dominant eigenvalue from each cospectral group. We use $|\mathcal{L}_{\mathbf{A}}|$ for the number of distinct ledgers and $|\lambda(\mathbf{A})|$ and $|\lambda(\mathcal{F}_{\mathbf{A}})|$ for the number of cospectral groups associated with the dominant eigenvalue.}} 
  \medskip
  \textbf{(i)}
  $\mathbf{A}$
  \\[0.4em]
  \begin{tabular}{@{}l r r l@{}}
    \toprule
    $\lambda_{max}(\mathbf{A})$ & $|\lambda(\mathbf{A})|$ & $|\mathbf{A}|$ & $|\mathcal{L}_{\mathbf{A}}|$ \\
    \midrule
    $-1$         & $1$ & $5$ & $1$\\
    $-1\pm i$    & $1$ & $2$  & $1$\\
    $-\tfrac12\pm i\tfrac{\sqrt3}{2}$ & $1$ & $4$ & $2$\\
    $0$          & $3$ & $21$ & $6$\\
    $\pm i$      & $1$ & $2$ & $1$\\
    $\tfrac12\pm i\tfrac{\sqrt3}{2}$ & $1$ & $4$  & $2$\\
    $\varphi^{-1}$ & $1$ & $4$ & $2$\\
    $1$          & $3$ & $27$ & $6$\\
    $1\pm i$     & $1$ & $2$ & $1$\\
    $\sqrt{2}$   & $1$ & $4$ & $2$\\
    $\varphi$    & $1$ & $4$ & $2$\\
    $2$          & $1$ & $2$ & $1$\\
    \bottomrule
  \end{tabular}

  \medskip
  \textbf{(ii)} 
  $\mathcal{F}_{\mathbf{A}}$ \\[0.4em]
  \begin{tabular}{@{}l r r l@{}}
    \toprule
    $\lambda_{\max}(\mathcal{F}_\mathbf{A})$ & $|\lambda(\mathcal{F}_\mathbf{A})|$ & $|\mathbf{A}|$ & $|\mathcal{L}_{\mathbf{A}}|$  \\
    \midrule
    $-\tfrac12\pm i\tfrac{\sqrt3}{2}$ & $1$ & $5$ & $1$\\
    $0$          & $3$ & $19$ & $5$\\
    $\tfrac12\pm i\tfrac{\sqrt3}{2}$ & $1$ & $5$ & $1$\\
    $0.819\ldots$ & $1$ & $8$ & $4$\\
    $1$          & $3$ & $18$ & $6$\\
    $1.380\ldots$ & $1$ & $8$ & $4$\\
    $\varphi$    & $1$ & $10$ & $2$\\
    $\sqrt{3}$   & $1$ & $8$  & $4$\\
    \bottomrule
  \end{tabular}
\end{table}

Ordering all $81$ patterns first by their $\lambda_{\max}(\mathbf{A})$ and then their corresponding $\lambda_{\max}(\mathcal{F}_{\mathbf{A}})$ for comparison yields little global coherence (Spearman $\rho\approx0.63$). However, the cospectral set with $\lambda_{\max}(\mathcal{F}_{\mathbf{A}}) = -\tfrac12\pm i\tfrac{\sqrt3}{2}$ and which was described above (form $\mathbf{d}$ and its cospectral set listed as the first five forms in Table 2) are the same forms that have $\lambda_{max}(\mathbf{A})=-1$, and which are the minimal dominant set of the $\lambda_{max}(\mathbf{A})$ in Table 3. Two further forms share $\operatorname{Re}(\lambda_{\max}(\mathbf{A}))=-1$ but are a separate cospectral class with $\lambda_{\max}(\mathbf{A})=-1\pm i$. For clarity, we define two sets. Let $\mathcal{A}_n=\{\mathbf{A}: min \operatorname{Re}(\lambda_{max}(A))\}$, with a further field distinction of $\mathcal{A}_n^{\mathbb{R}}
= \{\mathbf{A}\in\mathcal{A}_n : \lambda_{\max}(\mathbf{A})\in\mathbb{R}\}$
and $\mathcal{A}_n^{\mathbb{C}\setminus\mathbb{R}} = \{\mathbf{A}\in\mathcal{A}_n : \lambda_{\max}(\mathbf{A})\in\mathbb{C}\setminus\mathbb{R}\}$, and $\mathcal{G}_n=\{\mathcal{F}_{\mathbf{A}}: min \operatorname{Re}(\lambda_{\max}(\mathcal{F}_{\mathbf{A}})\}$. We can then extend the above to n=3,4 (n=5 is computationally intractable) as can be seen in Table 4. \\

\begin{table}[ht]
  \centering
  \caption{\textit{The number of distinct forms, ledgers, and set numbers for $\mathcal{G}_n$ and the field distinctions of  $\mathcal{A}_n$ for n=2,3,4.}}
    \footnotesize
  \renewcommand{\arraystretch}{1.12}
  \setlength{\tabcolsep}{5pt}
  \begin{tabular}{@{} r r r r rr @{}}
    \toprule
    $n$ & $|\mathbf{A}|$ & $|\mathcal{L}_{\mathbf{A}}|$ & $|\mathcal{G}_n|$ &
    $|\mathcal{A}_n^{\mathbb{R}}|$ & $|\mathcal{A}_n^{\mathbb{C}\setminus\mathbb{R}}|$ \\
    \midrule
    2 & 81 & 27 & 5 & 5 & 2 \\
    3 & 19{,}683 & 2{,}997 & 109 & 113 & 122 \\
    4 & 43{,}046{,}721 & 4{,}060{,}761 & 268{,}403 & 0 & 45{,}521 \\
    \bottomrule
  \end{tabular}
\end{table} 

From the exhaustive results that have been summarised in Table 4, the following can be stated: \\
\[
\mathcal{G}_2=\mathcal{A}_2^{\mathbb{R}},\qquad
\mathcal{G}_3\subsetneq\mathcal{A}_3^{\mathbb{R}},\qquad
\mathcal{G}_4\cap\mathcal{A}_4=\varnothing
\]\\
leading to the question of whether $\mathcal{G}_n\cap\mathcal{A}_n=\varnothing$, $\forall n>3$.\\

\begin{table}[ht]
  \centering
  \caption{\textit{Coefficient ledger counts for all $3^{n^2}$ forms for $n=2,3,4$. $(P,N)$ splits are different $(P_k,N_k)$ pairs sharing the same $D_k$.}}
  \label{tab:ledger-layers}
  \small
  \begin{tabular}{@{}clrrrr@{}}
    \toprule
    $n$ & coefficient & $|D_k|$ & $|P_k|$ & $|N_k|$ & $(P,N)$ splits \\
    \midrule
    2 & $D_1$ & $9$   & $4$  & $4$  & $0$ \\
    2 & $D_0$ & $9$   & $4$  & $4$  & $0$ \\
    \addlinespace
    3 & $D_2$ & $27$  & $8$  & $8$  & $0$ \\
    3 & $D_1$ & $297$ & $40$ & $56$ & $0$ \\
    3 & $D_0$ & $355$ & $56$ & $56$ & $37$ \\
    \addlinespace
    4 & $D_3$ & $81$ & $16$ & $16$ & $0$ \\
    4 & $D_2$ & $19{,}683$ & $832$ & $1{,}536$ & $2{,}916$ \\
    4 & $D_1$ & $1{,}919{,}157$ & $269{,}742$ & $269{,}742$ & $100{,}703$ \\
    4 & $D_0$ & $159{,}455$ & $34{,}752$ & $34{,}752$ & $49{,}029$ \\
    \bottomrule
  \end{tabular}
\end{table}

It is worth commenting on two further structural points. As described above with the four example cases ($\mathbf{a}$-$\mathbf{d}$), the minimal case (form $\mathbf{d}$), along with its cospectral group, is the one whose D-tuple consists solely of identities ($D_0=D_1=I_2$). For $n=3$, the D-tuple that is shared by the $109$ $\mathcal{G}_3$ all have $D_0=D_3=I_3$ and $D_1=2I_3$. However, the other cospectral set of 4 forms that are in $\mathcal{A}_3^{\mathbb{R}}$ have $D_0$ and $D_1$ that are neither the identity nor a multiple of the identity, but have off-diagonal values. There are no identities in the $D_k$ of the minimal sets for $n=4$. 

Table 5 lists the cases where for a given $D_k = P_k-N_k$, different $P_k$,$N_k$ result in the same $D_k$, that is, there are (P,N) splits. While this splitting already occurs in the $n=3$ case, none of the forms in $\mathcal{A}_3$ or $\mathcal{G}_3$ exhibit this attribute.  \\


\centerline {\bf References}
\medskip
\noindent [1] M. Gardner, W. Ashby. Connectance of large dynamic (cybernetic) systems: critical values for stability. Nature, 228(5273):784, 1970.
\medskip

\noindent [2] M. Thorne, E. Forgoston, L. Billings, and A-M. Neutel. Matrix scaling and tipping points. SIAM J. Appl. Dyn. Syst.,
20(2):1090-1103, 2021.
\medskip

\noindent [3] S. Gerschgorin, \"{U}ber die Abgrenzung der Eigenwerte einer Matrix. Izv. Akad. Nauk. USSR Otd. Fiz.-Mat, 6:749-754, 1931.
\medskip

\noindent [4] R. Descartes, La G\'{e}om\'{e}trie. (Paris), 1637. 
\medskip

\noindent [5] N. Obrechkoff, Sur un probl\`{e}me de Laguerre. C.R. Acd. Sci. (Paris), 177:223-235, 1923.
\medskip

\noindent [6] M. Thorne. Tipping cycles. Linear Algebra Appl., 646:43-53, 2022.
\medskip

\noindent [7] M. Fiedler. Algebraic connectivity of graphs. Czechoslovak Mathematical Journal, Vol 23, No. 2 298-305, 1973
\medskip

\noindent [8] N. Biggs Algebraic Graph Theory. Cambridge University Press, Cambridge, 1974
\medskip

\noindent [9] J. Maybee, J. Quirk. Qualitative problems in matrix theory. SIAM Rev., 11, pp. 30-51, 1969
\medskip

\noindent [10] C. Jeffries, V. Klee, P. van der Driessche. When is a matrix sign stable? Can. J. Math., 29, pp. 315-326, 1977

\end{document}